\documentclass{amsart}
    \usepackage{amsmath}
    \usepackage{amssymb,amsfonts}
    \usepackage[all,arc]{xy}
    \usepackage{enumerate}
    \usepackage{mathrsfs}
    \usepackage{amsthm}
    \usepackage{enumitem}
    \usepackage{extarrows}
    \usepackage{verbatim}
    \usepackage{dsfont}
    \usepackage{tikz-cd}
    \usepackage{tensor}
    \usepackage{mathtools}

    \newtheorem{thm}{Theorem}[section]
    \newtheorem{cor}[thm]{Corollary}
    \newtheorem{prop}[thm]{Proposition}
    \newtheorem{lem}[thm]{Lemma}

    \theoremstyle{definition}
    \newtheorem{defn}[thm]{Definition}
    \newtheorem{defns}[thm]{Definitions}

    \theoremstyle{remark}
    
    \newtheorem{rem}[thm]{Remark}

    \newcommand{\Z}{\mathbb{Z}}
    
    \newcommand{\C}{\mathbb{C}}
    \newcommand{\pP}{\mathbb{P}}

    \newcommand{\cO}{\mathcal{O}}

    \newcommand{\Hom}{\mathrm{Hom}}
    \newcommand{\Ext}{\mathrm{Ext}}
    
    \makeatletter
    
    \newcommand*{\rom}[1]{\expandafter\@slowromancap\romannumeral #1@}
    
    \makeatletter
    
    \let\c@equation\c@thm
    \makeatother
    \numberwithin{equation}{section}

\begin{document}

\title{Compactification of the moduli space of minimal instantons on the Fano 3-fold $V_4$}

\author{Xuqiang QIN}
\address{Department of Mathematics, Indiana University, 831 E. Third St., Bloomington, IN 47405, USA}
\email{qinx@iu.edu}


\begin{abstract}
    We study semistable sheaves of rank $2$ with Chern classes $c_1=0$, $c_2=2$ and $c_3=0$ on the Fano threefold $V_4$ of Picard number $1$, degree $4$ and index $2$. We show that the moduli space of such sheaves is isomorphic to the moduli space of semistable rank $2$, degree $0$ vector bundles on a genus $2$ curve. This also provides a natural smooth compactification of the moduli space of Ulrich bundles of rank $2$ on $V_4$.
\end{abstract}
    
\maketitle
\section{Introduction}

Instanton bundles first appeared in \cite{AHDM}  as a way to describe Yang-Mills instantons on a 4-sphere $S^4$. They provide extremely useful links between mathematical physics and algebraic geometry. The notion of mathematical instanton bundle was first introduced on $\pP^3$. By definition a mathematical instanton of charge $n$ is a stable rank $2$ vector bundle with on $\pP^3$ with Chern classes $c_1(E)=0, c_2(E)=n$, satisfying the instantonic vanishing condition:
    \begin{align*}
    h^1(E(-2))=0.
    \end{align*}
Since then the irreducibility \cite{T} and smoothness \cite{JV} of their moduli space were heavily investigated. Faenzi \cite{Fa} and Kuznetsov \cite{Ku2} generalized this notion to Fano threefolds, we recall
\begin{defn}\cite{Ku2}
    Let $Y$ be a Fano threefold of index $2$. An \emph{instanton bundle of charge $n$} on $Y$ is a stable vector bundle $E$ of rank $2$ with $c_1(E)=0, c_2(E)=n$, enjoying the instantonic condition: 
    \begin{align*}
        H^1(Y,E(-1))=0.
    \end{align*}
\end{defn}
We mention that the charge $c_2(E)\geq 2$ \cite[Corollary 3.2]{Ku2}. The instanton bundles of charge $2$ are called the \emph{minimal instantons}.\\

 In this paper, we will be concerned with minimal instantons and natural compactification of their moduli on the Fano threefold of Picard rank $1$, index $2$ and degree $4$, which we denote by $V_4$. Such a threefold is an intersection of two quadrics in $\pP^5$. The moduli space of minimal instanton bundles on $V_4$ was discussed in \cite{Ku2} and it was shown to be an open subscheme of the moduli space of rank $2$ even degree bundles on a genus $2$ curve $C$ which is naturally associated to $V_4$ (see \cite[Theorem 5.10]{Ku2}).\\

On the other hand, Ulrich bundles are defined as vector bundles on a smooth projective variety $X$ of dimension $d$ such that 
\begin{align*}
        H^*(X,E(-t))=0
    \end{align*}
for all $t=1,\ldots,d$. They first appeared in commutative algebra and entered the world of algebraic geometry via \cite{ES}. The existence and the moduli space of Ulrich bundles provide great amount of information about the original variety. For example, in the case when $X$ is a smooth hypersurface, the existence of Ulrich bundles is equivalent to the fact that $X$ can be defined set-theoretically by a linear determinant \cite{Beau}. Inspired by \cite{Ku2}, Lahoz, Macr\`i and Stellari \cite{LMS}\cite{LMS2}, studied moduli spaces of Ulrich bundles on cubic threefolds and fourfolds using derived categories. In a recent paper \cite{CKL}, the moduli space of stable Ulrich bundles of rank $r$ was shown to be an open subscheme of the moduli space of rank $r$ degree $2r$ vector bundles on $C$.
We will see that on $V_4$, the minimal instanton bundles and the Ulrich bundles of rank $2$ differ only by a twist by $\cO_{V_4}(1)$. Thus they share the same moduli space and compactifications.\\

Our first result states that the stable rank 2 vector bundles on $V_4$ with Chern classes $c_1=0,c_2=2$ automatically satisfy the instantonic vanishing condition, and hence are minimal instantons.
\begin{thm}
    Let $E$ be a stable rank $2$ vector bundle on $V_4$, with $c_1(E)=0$ and $c_2(E)=2$, then $H^1(V_4,E(n))=0$ for all $n\in \Z$. In particular, $E$ is a minimal instanton bundle.
\end{thm}

In light of this result, we consider the moduli space of semistable rank $2$ sheaves with Chern classes $c_1=0$, $c_2=2$  and $c_3=0$ on $V_4$ as a natural compactification of the moduli space of minimal instanton bundles.\ 

We also mention that similar phenomenon was observed on cubic threefolds (See \cite[Theorem 2.4]{D}), but the proof used properties of the cubic.\

Our next result is the classification of semistable rank $2$ sheaves with Chern classes $c_1=0$, $c_2=2$  and $c_3=0$ on $V_4$. \cite{D} classified sheaves with same numerical conditions on cubic threefolds and proved that their moduli space is isomorphic to the blow-up of the intermediate Jacobian in the Fano surface of lines. We follow his method and prove that a parallel classification happens on $V_4$.
\begin{thm}
     Let $E$ be a semistable rank $2$ sheaf with Chern classes $c_1=0$, $c_2=2$  and $c_3=0$ on $V_4$. If $E$ is stable, then either $E$ is locally free or $E$ is associated to a smooth conic $Y\subset V_4$ such that we have the exact sequence:
     \begin{align*}
         0\to E\to H^0(\theta(1))\otimes \cO_{V_4}\to \theta(1)\to 0,
     \end{align*}
     where $\theta$ is the theta-characteristic of $Y$.\\
         If $E$ is strictly semistable, then  $E$ is the extension of two ideal sheaves of lines.
\end{thm}
Recall that a theta-characteristic of a non-singular curve $Y$ is a line bundle $L$ such that $L^{\otimes 2}$ is the canonical bundle. In the case when $Y$ is a smooth conic, $Y\simeq \pP^1$ and a theta-characteristic is just the negative of the ample generator in the Picard group of $Y$.

Unfortunately, the method to study the moduli space in \cite{D} does not transfer well to $V_4$. However, we note that $\mathcal{D}^b(V_4)$ has a semi-orthogonal decomposition:
\begin{align*}
    \mathcal{D}^b(V_4)=\langle \mathcal{B}_{V_4},\cO_{V_4},\cO_{V_4}(1)\rangle
\end{align*}
There is a genus $2$ smooth curve $C$ naturally associated to $V_4$ such that there is a natural choice of equivalence of triangulated categories  $\Phi:\mathcal{D}^b(C)\simeq \mathcal{B}_{V_4}$ (see Section 2.2 for the precise definition). This functor is Fourier-Mukai by Orlov's result. Our second result connects semistable rank $2$ sheaves with Chern classes $c_1=0$, $c_2=2$  and $c_3=0$ with rank $2$ bundles on $C$.
\begin{thm}
    Let $E$ be a semistable rank $2$ sheaf with Chern classes $c_1=0$, $c_2=2$  and $c_3=0$ on $V_4$, then $E\in \mathcal{B}_{V_4}$ and $\Phi^*(E)[-1]$ is a rank $2$ semistable vector bundle of degree $0$ on $C$.
    Moreover, if $E$ is stable (strictly semistable), then the vector bundle $\Phi^*(E)[-1]$ is stable (strictly semistable).
\end{thm}
Using this relation, we construct a morphism between the moduli space $M_{V_4}$ of semistable rank $2$ sheaves with Chern classes $c_1=0$, $c_2=2$  and $c_3=0$ on $V_4$ and the moduli space $M_C$ of semistable vector bundles on $C$ of rank $2$ and degree $0$ and prove it is an isomorphism.
\begin{thm}
    There exists a morphism $\psi \colon M_{V_4}\to M_C$ which is an isomorphism.   As a result, the moduli space of semistable rank $2$ sheaves with Chern classes $c_1=0$, $c_2=2$  and $c_3=0$ on $V_4$ is a $\pP^3$-bundle over the Jacobian of $C$, thus a smooth projective variety of dimension $5$.
\end{thm}

We mention some related work. On projective space $\pP^3$, the question of how instanton bundles degenerate has been intensely studied. Maruyama and Trautmann \cite{MT} were the first to consider limits of instantons. Instanton sheaves of charge $n$ on $\pP^3$ are first defined by Jardim \cite{J} as torsion free sheaves $E$ with $c_1(E)=0$, $c_2(E)=n$, $c_3(E)=0$ satisfying
$$h^0(E(-1))=h^1(E(-2))=h^2(E(-2))=h^3(E(-3))=0.$$ (See also \cite{J} for a study of instanton sheaves on general $\pP^n$.) Gargate and Jardim \cite{GJ} showed (among other properties) the singular locus of $E$, i.e. the quotient of the double dual of $E$ by $E$, has pure dimension $1$.
Jardim, Markushevich, Tikhomirov \cite{JMT}  considered  the boundary of moduli space of instanton bundles of charge $n$ on $\pP^3$ in the moduli space of semistable rank $2$ sheaves with Chern classes $c_1=0,c_2=n$ and $c_3=0$. They showed that the boundary contains $n$ irreducible divisors $\mathcal{D}(m,n)$ where $m=1,2\ldots,n$, whose generic points represent instanton sheaves singular along a normal rational curves of degree $m$. Moreover, these divisors only suffice to compactify the moduli space when $n=1$ is the minimal charge. We note that the boundary sheaves in our compactification do not satisfy instantonic vanishing conditions (translate to $h^0(E)=h^1(E(-1))=h^2(E(-1)=h^3(E(-2))=0$ on Fano threefolds of index $2$). But they do satisfy the property that they are singular on rational curve of degree $\leq 2$ and suffice to compactify the moduli space of minimal instantons. \  

Faenzi \cite{Fa} and Kuznetsov \cite{Ku2} generalized the notion of instanton bundles to Fano threefolds and studied instanton bundles on Fano threefolds of index $2$ and Picard rank $1$ using techniques in derived categories. Properties of instanton bundles and their moduli spaces on degree $4,5$ cases were investigated in detail in \cite{Fa} and \cite{Ku2}.
The author \cite{Q} studied compactification of the moduli space of  minimal instantons on the degree $5$ Fano threefold of index $2$ by associating such sheaves with quiver representations via the Kuznetsov component.\ 

Casnati, Coskun, Genc and Malaspina \cite{CCGM}  proposed a general definition of instanton bundles with given charge on any Fano threefold, unifying all previous notions. They studied instanton bundles on the blow up $\widetilde{\pP^3}$ of $\pP^3$ at a point in the same paper. Henni \cite{He} examined instanton sheaves and their moduli on $\widetilde{\pP^3}$.\\

This paper is organized as follows. In the second section the reader can find some preliminary definitions and results that are used throughout the paper. In the third section we classify the semistable rank $2$ sheaves with Chern classes $c_1=0$, $c_2=2$  and $c_3=0$, showing the parallel result to that holding for cubic threefolds. In the fourth section we connect such sheaves to vector bundles on $C$ using derived category. In the last section we describe the compactification of the moduli space of instantons on $V_4$.

\subsection*{Notation and conventions}
\begin{itemize}
    \item We work over the complex numbers $\C$.
    \item Let $E$ be a sheaf on $V_4$. We use $H^i(E)$ to denote $H^i(V_4,E)$ for simplicity. Also we use $h^i(E)$ to denote the dimension of $H^i(V_4,E)$ as a complex vector space.
    \item  $M_{V_4}$ denotes the moduli space  of semistable rank $2$ sheaves with Chern classes $c_1=0$, $c_2=2$  and $c_3=0$ on $V_4$.
    \item $M_C$ denotes the moduli space of semistable vector bundles of rank $2$ and degree $0$ on the associated curve $C$.
    \item Let $F$ be a sheaf with certain characterization, we will use $[F]$ to denote the point it corresponds to in the moduli space.
\end{itemize}
\subsection*{Acknowledgement}
The author would like to thank his advisor Valery Lunts for constant support and inspiring discussions. He would like to thank Izzet Coskun, St\'ephane Druel, Alexander Kuznetsov, Michael Larsen, Emanuele Macr\`i, Shizhuo Zhang for useful discussions. He is very grateful to the anonymous referees for many helpful comments.
\section{Preliminaries}
    
\subsection{Derived Categories} Let $X$ be an algebraic variety, we use $\mathcal{D}^b(X)$ to denote the derived categories of coherent sheaves on $X$. For objects $F,G\in \mathcal{D}^b(X)$, we denote $\Ext^p(F,G)=\Hom(F,G[p])$ and $\Ext^\bullet(F,G)=\oplus_{p\in\Z}\Ext^p(F,G)[-p]$. Recall that a sequence of full admissible triangulated subcategories of a triangulated category $\mathcal{T}$
\begin{align*}
    D_1,\ldots,D_n\subset \mathcal{T}
\end{align*}
is semi-orthogonal if for all $j>i$
\begin{align*}
    D_i\subset D_j^{\perp},
\end{align*}
where $D_j^{\perp}$ is the full subcategory of objects $C\in\mathcal{T}$ such that $\Hom(B,C)=0$ for all objects $B\in D_j$. 
Such a sequence defines a semi-orthogonal decomposition of $\mathcal{T}$ if the smallest full subcategory of $\mathcal{T}$ containing $D_1,\ldots, D_n$ is itself, in this case we use the notion $\mathcal{T}=\langle D_1,\ldots,D_n\rangle$. An easy way to produce a semi-orthogonal decomposition is by using exceptional objects or collections.
\begin{defns}
    An object $F\in \mathcal{T}$ is called exceptional if $\Ext^\bullet(F,F)=\C$. A collection of exceptional objects $F_1,\ldots,F_n$ is called an exceptional collection if $\Ext^p(F_j,F_i)=0$ for all $j>i$ and all $p\in \Z$.
\end{defns}

On a smooth Fano threefold $V$ of index $2$, Kodaira vanishing theorem implies:
    \begin{align*}
        H^i(V,\cO_V)=0
    \end{align*}
    for all $i>0$. Thus all line bundles on $V$ are exceptional objects. Moreover, we can check $\{\cO_{V},\cO_{V}(1)\}$ is an exceptional collection using the fact that $V$ has index $2$. We denote their left orthogonal complement by $\mathcal{B}_{V}$ and obtain the following semi-orthogonal decomposition:
    \begin{align*}
        \mathcal{D}^b(V)=\langle \mathcal{B}_{V},\cO_{V},\cO_{V}(1)\rangle.
    \end{align*}
    Note an object $F\in \mathcal{D}^b(V)$ is in $\mathcal{B}_{V}$ if and only if 
    \begin{align*}
        &\Hom(\cO_{V},F[i])=0\\
        &\Hom(\cO_{V}(1),F[i])=0
    \end{align*}
    for all $i\in \Z$.
    
\subsection{Vector bundles}
In this section we recall several useful results about vector bundles on smooth projective varieties.
\begin{prop}\cite[Section 1]{H1}
  Let $X$ be a smooth projective variety of dimension at least $2$ and $E$ a vector bundle of rank $2$ on $X$. Suppose there exists a global section of $E$ whose zero locus $Y$ is of pure codimension $2$, then we have an exact sequence:
  \begin{align*}
      0\to \cO_X\to E\to I_Y\otimes \mathrm{det}(E)\to 0.
  \end{align*}
\end{prop}
    
\begin{thm}[The Serre construction]\cite[Section 1]{Ar}
    Suppose $X$ is a smooth projective variety of dimension at least $3$. Let $L$ be a invertible sheaf such that $h^1(L^{-1})=0$ and $h^2(L^{-1})=0$ and $Y\subset X$ a closed subscheme of pure codimension $2$. We have an isomorphism $\Ext^1(I_Y\otimes L,\cO_X)=H^0(\cO_Y)$. The subscheme $Y$ is the zero locus of a section of a vector bundle of rank $2$ with determinant $L$ if and only if $Y$ is locally complete intersection and $\omega_Y=(\omega_X\otimes L)|_Y.$
\end{thm}    

We recall the following useful result:
\begin{prop}[Mumford-Castelnuovo Criterion]\cite[Lemma 5.1]{FGA}\label{mumford}
  Let $F$ be a coherent sheaf on a projective variety $X$. Suppose $h^i(X,F(-i))=0$ for all $i\geq 1$, then $h^i(X,F(k))=0$ for all $i\geq 1$ and $k\geq -i$. Moreover $F$ is generated by global sections.
\end{prop}

\subsection{Fano 3-fold $V_4$}
(See also \cite[Section 5.1]{Ku2})
A Fano threefold of Picard rank $1$, index $2$ and degree $4$ is a smooth intersection of two quadrics in $\pP^5$. We let $V_4$ be such a threefold. There is a smooth curve $C$ of genus $2$ associated to $V_4$. We briefly recall its construction. Let $V$ be a complex vector space of dimension $6$ and $A$ a vector space of dimension $2$. A pair of quadrics gives a map $A\to S^2V^*$, so $\pP(A)$ parametrizes a family of quadrics in $\pP(V)$. There are $6$ degenerate quadrics in this family, giving $6$ points on $\pP(A)\simeq \pP^1$. $C$ is defined to by the double cover of $\pP(A)$ ramified at the $6$ points. Clearly $C$ is a curve of genus $2$. We use $\tau:C\to C$ to denote its hyperelliptic involution.\\

By looking at the spinor bundles on quadrics in $\pP(A)$, one can show there is a vector bundle $\mathcal{S}$  of rank $2$ on $C\times V_4$. The Fourier-Mukai functor with kernel $\mathcal{S}$ connects $C$ with $V_4$ in the following way:
\begin{thm}\cite[Theorem 2.7]{BO}
    The Fourier-Mukai functor $\Phi_\mathcal{S}:\mathcal{D}^b(C)\to \mathcal{D}^b(V_4)$ provides an equivalence of $\mathcal{D}^b(C)$ onto $\mathcal{B}_{V_4}$.
\end{thm}
From now on, we use $\Phi$ to denote $\Phi_\mathcal{S}$ for simplicity.\

Another way to understand the relation between $C$ and $V_4$ and the vector bundle $\mathcal{S}$ is due to Mukai. It is shown that $V_4$ is the moduli space of rank $2$ bundles on $C$ with a fixed determinant $\xi$ of odd degree. Then $\mathcal{S}$ is the universal family. We follow \cite{Ku2}'s convention and assume $\mathrm{deg}\xi=1$. Then
\begin{align*}
    \mathrm{det}(\mathcal{S})=\xi\boxtimes\cO_{V_4}(-1).
\end{align*}

We can also describe the Fano variety of lines $F(V_4)$ using this functor. First note it is straightforward to check that the ideal sheaf $I_l$ of a line is an object in $\mathcal{B}_{V_4}$. Then
\begin{thm}\cite[Lemma 5.5]{Ku2}\label{ideal}
    There is an isomorphism $G: F(V_4)\xrightarrow{\sim}\mathrm{Pic}^0(C)$ given by
    \begin{align*}
        G(l)=\Phi^{-1}(I_{l}[-1]).
    \end{align*}
    In particular $\Phi^{-1}(I_{l}[-1])$ is a degree $0$ line bundle on $C$.
\end{thm}
This result will be crucial in our analysis of strictly semistable sheaves.\\

Finally we provide some topological information that will be useful later on. Let $[h],[l],[p]$ be the class of a hyperplane section, a line and a point respectively. Then
\begin{align*}
     H^2(V_4,\Z)\simeq \Z[h], \quad H^4(V_4,\Z)\simeq \Z[l], \quad H^6(V_4,\Z)\simeq \Z[p]
\end{align*}
with $h\cdot l=p,h^2=4l,h^3=4p$.\\
The natural embedding $V_4\subset \pP^5$ provides a very ample divisor $\cO_{V_4}(1)$. We will always use this polarization when talking about stability. A general section of $|\cO_{V_4}(1)|$ is a del-Pezzo surface of degree $4$, hence isomorphic to $\pP^2$ blown up at $5$ points in general position.\\
We compute the Todd class of $V_4$
\begin{align*}
    \mathrm{td}(\mathcal{T}_{V_4})&=1+h+\frac{7}{3}l+p.
\end{align*}
We also recall the Grothendieck-Riemann-Roch for the functor $\Phi$.
\begin{lem}\cite[Lemma 5.2]{Ku2}\label{RR}
    For any $F\in \mathcal{D}^b(C)$ we have
    \begin{align*}
        \mathrm{ch}(\Phi(F))=(2\mathrm{deg}(F)-\mathrm{rk}(F))-\mathrm{deg}(F)h+\mathrm{rk}(F)l+\frac{\mathrm{deg}(F)}{3}p.
    \end{align*}
\end{lem}

\subsection{Stability of sheaves} Let $X$ be a smooth projective variety of dimension $n$ and $\cO_X(1)$ a fixed ample line bundle. Let $E$ be a coherent sheaf of rank $r$, then the slope of $E$ is defined as:
\begin{align*}
    \mu(E)=\frac{c_1(E)\cdot c_1(\cO_X(1))^{n-1}}{r\cdot c_1(\cO_X(1))^n }.
\end{align*}
The sheaf $E$ is called \emph{(semi)stable} if it is torsion free and for any torsion free subsheaf $F\subset E$, we have 
\begin{align*}
    \frac{\chi(F(n))}{\mathrm{rk}(F)}(\leq)<\frac{\chi(E(n))}{\mathrm{rk}(E)}
\end{align*}
for $n\gg0$. \\
The sheaf $E$ is called \emph{$\mu$-(semi)stable} if it is torsion free and for any torsion free subsheaf $F\subset E$, we have 
\begin{align*}
    \mu(F)(\leq)<\mu(E).
\end{align*}
We have the following implications:
\begin{align*}
    \text{$\mu$-stable}\Rightarrow \text{stable} \Rightarrow \text{semistable}\Rightarrow\text{$\mu$-semistable}
\end{align*}

When $X$ is a smooth projective curve, any torsion free sheaf is locally free. We have the following criterion for semistability which is due to Faltings:
\begin{lem}\cite[Exercise 2.8]{P}\label{scurve}
    Let $F,G$ be vector bundles on a curve $X$ such that $H^i(F\otimes G)=0$ for $i=0,1$, then both $F$ and $G$ are semistable.
\end{lem}
Let $C$ be the associated curve of $V_4$. Then $C$ has genus $2$. We use $M_C$ to denote the moduli space of semistable vector bundles of rank $2$ and degree $0$ on $C$. $M_C$ was studied in detail in \cite{NR}. Suppose $J^1$ is the moduli space of line bundles of degree $1$. Then $C$ is naturally embedded in $J^1$ as a divisor. We denote the corresponding divisor by $\Theta$. To understand $M_C$, it suffices to understand the $3$-dimensional subvariety $S\subset M_C$ consisting of bundles with trivial determinant. \cite{NR} showed $S$ is naturally isomorphic to the projective space associated to $H^0(J^1,2\Theta)$. From this it is not hard to conclude:
\begin{thm}\cite[Theorem 7.3]{NR}
    $M_C$ is canonically isomorphic to the space of positive divisors on $J^1$
algebraically equivalent to $2\Theta$. In particular, $M_C$ is a projective bundle over the Jacobian of $C$.
\end{thm}
As a consequence, we have
\begin{cor}
     $M_C$ is a non-singular projective algebraic variety of dimension $5$.
\end{cor}

\subsection{Instanton bundles and Ulrich bundles} Let $Y$ be a Fano threefold of index $2$. By definition an minimal \emph{instanton bundle} is a stable vector bundle $E$ of rank $2$ with Chern classes $c_1(E)=0$, $c_2(E)=2$, enjoying the instantonic condition
\begin{align*}
    H^1(Y,E(-1))=0.
\end{align*}
We will see in Theorem \ref{auto} that the instantonic condition is automatically satisfied on $V_4$. We use $M_{V_4}$ to denote the moduli space of semistable sheaves of rank $2$  with Chern classes $c_1(E)=0$, $c_2(E)=2$ and $c_3(E)=0$. It is clear that the moduli space of minimal instanton bundles $M^{inst}(V_4)$ is an open subscheme of $M_{V_4}$.\

We also recall two equivalent definitions of an Ulrich bundle.
\begin{defns}
    Let $X\subset \pP^N$ be a smooth projective variety of dimension $n$ and degree $d$. An \emph{Ulrich bundle} $E$ is a vector bundle on $X$ satisfying 
    \begin{align*}
        H^*(X,E(-t))=0
    \end{align*}
    for all $t=1,\ldots,n$. Equivalently, it is a vector bundle of rank $r$ satisfying
    \begin{align*}
        H^i(X,E(t))=0
    \end{align*}
    for all $t\in \Z$ and $0<i<n$ and having Hilbert polynomial $P_E(t)=dr\binom{n+t}{t}$
\end{defns}
We list a few well-known facts about Ulrich bundles that will be useful to us (see \cite[Section 3,4]{Beau}):
\begin{itemize}
    \item There are no Ulrich line bundles on a variety $X$ with $\mathrm{Pic}(X)=\Z\cO_X(1)$ and degree $d>1$.
    \item An Ulrich bundle is semistable. If it is not stable, it is an extension of Ulrich bundles of smaller ranks. 
\end{itemize}

To see the relation between minimal instanton bundles and Ulrich bundles of rank $2$ on $V_4$, we first recall the following result:
\begin{prop}\cite[Proposition 3.4]{CKL}
    Let $E$ be an Ulrich bundle of rank $r$ on $V_4$, then $\mu(E)=1$.
\end{prop}
As a result, if $E$ is a rank $2$ Ulrich bundle, $E(-1)$ is a semistable rank $2$ bundle with Chern class $c_1=0$. Moreover, we have the Hilbert polynomial
\begin{align*}
    P_{E(-1)}(t)=8\binom{t+2}{3}.
\end{align*}
Using Riemann-Roch, it is not hard to see $c_2(E(-1))=2$. 
Combine this with the following result:
\begin{lem}\cite{Ku2}
    Let $E$ be an minimal instanton bundle on a Fano threefold of index $2$. Then 
    \begin{align*}
            H^*(E(t))=0
    \end{align*}
for $t=0,-1,-2$.
\end{lem}
We obtain
\begin{prop}
      A vector bundle $E$ on $V_4$ is an minimal instanton bundle if and only if $E(1)$ is a stable Ulrich bundle of rank 2.
\end{prop}
If we use $M_2^{sU}$ to denote the moduli space of stable Ulrich bundles of rank $2$. It follows immediately from this result that:
\begin{align*}
   M^{sU}_2\simeq M^{inst}(V_4) 
\end{align*}

To describe $M^{inst}(V_4)$, we first recall that by definition a second Raynaud bundle on $C$ is the (shift of the) Fourier-Mukai transform of
the bundle $\cO_{\mathrm{Pic}(C)}(-2\Theta)$ with the kernel given by the Poincare bundle, where $\Theta$ is the theta divisor of $\mathrm{Pic}(C)$ (see \cite{Ku2}\cite{P}). It is a semistable rank $4$ vector bundle of degree $4$ on $C$.
Kuznetsov gave the following description of the moduli space $M^{inst}_n(V_4)$ of instanton bundles of charge $n$ 
\begin{thm}\cite[Theorem 5.10]{Ku2}
   Let $\mathcal{R}$ be a second Raynaud bundle. The moduli space $M^{inst}_n(V_4)$ of instantons of charge $n$ is isomorphic to the moduli space of simple vector bundles $\mathcal{F}$ on $C$ of rank $n$ and degree $0$ such that 
   \begin{align*}
       &\mathcal{F}^*\simeq \tau^*\mathcal{F},\\
       &H^0(C,\mathcal{F}\otimes \mathcal{S}_y)=0 \text{ for all } y\in V_4,\\
       &\dim \Hom(\mathcal{F},\mathcal{R})=\dim\Ext^1(\mathcal{F},\mathcal{R})=n-2.
   \end{align*}
\end{thm}
\begin{rem}
Apply this theorem to $n=2$ and note that in this case the last equation shows that $\mathcal{F}$ is semistable by Lemma \ref{scurve}, so we see that
\begin{align*}
   M^{sU}_2\simeq M^{inst}(V_4) \subset M_C.
\end{align*}
See also \cite[Theorem 4.14]{CKL} for a similar result from the perspective of Ulrich bundles.
\end{rem}

\subsection{Curves and surfaces of low degrees}
In this section we recall some results about varieties of (almost) minimal degrees in projective spaces. Recall a variety $V\subset \pP^n$ is said to be \emph{non-degenerate} if $V$ is not contained in any hyperplane in $\pP^n$.\\

Regarding the degree of a curve we have first the classical Castelnuovo Bound:
\begin{thm}[Castelnuovo Bound]
    Let $C\subset \pP^n$ be an irreducible non-degenerate smooth curve of degree $d$ and genus $g$. Let $m,\epsilon$ be the quotient and remainder when dividing $d-1$ by $n-1$, and $\pi(d,n)=(n-1)m(m-1)/2+m\epsilon$, then
    \begin{align*}
        g\leq \pi(d,n).
    \end{align*}
\end{thm}
In fact this bound is sharp, and extremal curves (curves with $g=\pi(d,n)$) can be explicitly described (See Chapter $3$ of \cite{ACGH}). This result was improved by Eisenbud and Harris:
\begin{thm}\cite[Theorem 3.15]{EH} \label{pi1}
    With the same notation as above, for any $d$ and $n\geq 4$, set $m_1$,$\epsilon_1$ to be the quotient and remainder when dividing $d-1$ by $n$. Let $\mu_1=1$ if $\epsilon_1=n-1$ and $0$ otherwise.
    Let
    \begin{align*}
        \pi_1(d,n)=\binom{m_1}{2}n+m_1(\epsilon_1+1)+\mu_1,
    \end{align*}
    then
    \begin{enumerate}
        \item if $g>\pi_1(d,n)$ and $d\geq 2n+1$, then $C$ lies on a surface of degree $n-1$; and 
        \item if $g=\pi_1(d,n)$ and $d\geq 2n+3$, then $C$ lies on a surface of degree $n$ or $n-1$.
    \end{enumerate}
\end{thm}
We now look at surfaces of small degrees. It is well known that for an irreducible, non-degenerate variety $M\subset \pP^n$ of dimension $m$, the minimal degree is $n-m+1$. (See \cite{GH} Section 1.3). Minimal surfaces are well understood.
\begin{thm}\cite[Section 4.3]{GH}
    Every non-degenerate irreducible surface of degree $n-1$ in $\pP^n$ is either a rational normal scroll or the Veronese surface in $\pP^5$.
\end{thm}
The case of degree $n$ surface in $\pP^n$ is a little more involved. 
\begin{thm}\cite[Theorem 8]{N}
    Every non-degenerate irreducible surface $S$ of degree $m\neq 8$ in $\pP^m$ is one of the following:
    \begin{enumerate}
        \item Projection of a minimal surface in $\pP^{m+1}$.
        \item A del-Pezzo surface with isolated double points.
        \item A cone with a smooth elliptic base curve.
    \end{enumerate}
\end{thm}
\begin{rem}
    When $m=8$ we have two more kinds of surfaces, we will not need this case. The reader can find more details in \cite[Theorem 8]{N}.
\end{rem}

\section{Classification of semistable rank $2$ sheaves with Chern classes $c_1=0$, $c_2=2$  and $c_3=0$ on $V_4$}
We first show that for stable rank $2$ vector bundles with Chern classes $c_1=0$, $c_2=2$ on $V_4$, the instantonic vanishing condition is satisfied..
\begin{lem}
    Let $S\subset \pP^4$ be a del Pezzo surface of degree $4$ and $E$ a $\mu$-semistable vector bundle of rank $2$ with Chern classes $c_1(E)=0$ and $c_2(E)=2$. If $h^0(E)=0$, then $h^1(E(n))=0$  for $n\in \Z$ and $h^2(E(n))=0$ for $n\geq -1$. If $h^0(E)\neq 0$, then $h^0(E)=1$, $h^1(E(n))=0$ for $n\leq -2$ and $n\geq 1$, $h^1(E(-1))=h^1(E)=1$ and $h^2(E(n))=0$ for $n\geq 0$.
\end{lem}
\begin{proof}
    See \cite[Lemma 2.2]{D}
\end{proof}
    \begin{thm}\label{auto}
        Let $E$ be a stable rank $2$ vector bundle on $V_4$, with $c_1(E)=0$ and $c_2(E)=2$. Then $H^1(V_4,E(n))=0$ for all $n\in \Z$. In particular, $E$ is a minimal instanton bundle.
    \end{thm}
\begin{proof}
    Let $S\in |\cO_{V_4}(1)|$ be a general hyperplane section of $V_4$. Then $E_S$ is $\mu$-semistable with respect to the polarization $\cO_S(1)$ \cite[Theorem 3.1]{M}.\\
    Suppose $h^0(E_S)=0$. Consider the short exact sequence:
    \begin{align*}
        0\to E(n-1)\to E(n)\to E_S(n)\to 0.
    \end{align*}
    Since $h^1(E_S(n))=0$ for $n\in \Z$, we have $h^1(E(n))\leq h^1(E(n-1))$. Thus $h^1(E(n))=0$ for all $n\in\Z$ since $h^1(E(n))=0$ for $n\ll 0$.\\
    Suppose $h^0(E_S)\neq 0$, we will try to get a contradiction. We claim $E(2)$ is generated by global sections. Using the above exact sequence, we obtain $h^1(E(-n))=0$ for $n\geq 2$. Note $h^2(E)=h^1(E(-2))=0$ and $h^3(E)=h^0(E(-2))=0$. Since $\chi(E)=0$, we have $h^1(E)=0$ and the exact sequence:
    \begin{align}\label{C1}
        0\to E\to E(1)\to E_S(1)\to 0.
    \end{align}
    gives $h^1(E(1))=h^1(E_S(1))=0$. We have then $h^3(E(-1))=h^0(E(-1))=0$. Thus $E(2)$ is generated by global sections by Mumford-Castelnuovo criterion (Proposition \ref{mumford}).\\
    If there exists a nowhere vanishing section of $E(2)$, then $E$ is isomorphic to $\cO_{V_4}(2)\oplus\cO_{V_4}(-2)$, 
    which is absurd. We have then an exact sequence:
    \begin{align}\label{C2}
        0\to \cO_{V_4}(-4)\to E(-2)\to I_Y\to 0
    \end{align}
    where $Y\in V_4$ is a smooth curve of degree $c_2(E(2))=18$. We have $h^1(I_Y)=0$, so the curve $Y$ is connected. We have $\omega_Y=\cO_Y(2)$ and $g(Y)=19$. Finally, the curve $Y$ is non-degenerate since $E$ is stable. Using (\ref{C1}), (\ref{C2}), it is not hard to find $h^0(\cO_Y(1))=7$. Thus the curve $Y$ is the projection to $\pP^5$ of a non-degenerate curve in $\pP^6$ with degree $18$ and genus $19$. The next lemma shows this leads to a contradiction.
\end{proof}
\begin{lem}
    Let $\tilde{Y}\subset \pP^6$ be a non-degenerate curve of degree $18$ and genus $19$. Let $O\notin \tilde{Y}$ be a point such that the projection from $O$ induces an embedding of $\tilde{Y}$ into $\pP^5$. Then the image $Y$ of $\tilde{Y}$ cannot lie in the intersection $V_4$ of two quadrics in $\pP^5$.
\end{lem}
\begin{proof}
    We first apply  Theorem \ref{pi1} to $\tilde{Y}$. In this case, we have $d=18$ and $n=6$, then $m_1=2$ and $\epsilon_1=5=6-1$. So $\mu_1=1$. We compute $\pi_1(18,6)=6+2\times 6+1=19=g$. Thus by the second part of Theorem \ref{pi1}, $\tilde{Y}$ lies on a surface $S$ of $5$ or $6$ in $\pP^6$. \
    
    Degree $5$ surfaces in $\pP^6$ are the images of $\mathbb{F}_{1+2k}$, $k=0,1,2$ of the morphisms $\tau_k$ induced by complete linear systems $|C_0+(3+k)f|$, where $C_0$ is the unique section with $C_0^2=-1-2k$ and $f$ is a general fibre. We have $\tilde{Y}\in |3C_0+(12+3k)f|$. Note when $k=0,1$, $\tau_k$ is an closed embedding while $\tau_2$ contracts the section $C_0$ and the image is a cone. \
    
    Degree $6$ surfaces in $\pP^6$ are:
    \begin{enumerate}
        \item A cone over smooth elliptic curve of degree $6$ in $\pP^5$.
        \item A del Pezzo surface of degree $6$ with possibly isolated double points.
        \item A projection of $\mathbb{F}_{2k}\subset \pP^7$ (k=0,1,2,3), embedded via the complete linear system $|C_0+(3+k)f|$, from a point outside of the surface.
    \end{enumerate}
    Let $\pi$ be the projection from $O\in \pP^6$. Let $\pi(S)$ be the image of $S$ under the rational map $\pi$. If $\pi(S)$ is one dimensional, then $S$ is a cone with base $\tilde{Y}$, which is absurd, since $S$ can only be a cone over a rational or elliptic curve by the classification above. Thus $\pi(S)$ is two dimensional. If $S$ is a cone, then its vertex is different from the point $O$.\
    
    Now suppose $Y\subset \pP^5$ is contained in $V_4$ and use $\overline{V_4}\subset \pP^6$ to denote the cone with vertex $O$ and base $V_4$.\
    
    Suppose $\overline{V_4}$ does not contain the surface $S$. Recall $V_4$ is the complete intersection of two smooth quadrics $Q_0$,$Q_1$. Use $\overline{Q_i}$ to denote the cone with vertex $O$ and base $Q_i$. Then $\overline{V_4}$ is the intersection of $\overline{Q_i}$'s. Then $S$ is not contained in at least one of the $\overline{Q_i}'s$, say $\overline{Q_0}$. Then $\overline{Q_0}$ cut the surface $S$ at a curve of degree $10$ or $12$, which cannot contain $\tilde{Y}$, contradiction. So $S\subset \overline{V_4}$. We thus also have $\pi(S)\subset V_4$.\
    
    Suppose $O\notin S$. Denote the degree of $\pi$ by $d$, then $\pi(S)$ is a surface of degree $5/d$ or $6/d$. On the other hand, $\pi(S)$ corresponds to a Cartier divisor $\cO_{V_4}(l)$ where $l$ is an integer, thus has degree $4l$. We have now a contradiction since $4dl=5$ or $4dl=6$ have no integral solutions.\
    
    Suppose $O\in S$. If $S$ has degree $5$, then $S$ is one of the surfaces $\tau_k(\mathbb{F}_{1+2k})$ for $k=0,1,2$. The fibre $f$ passing through $O$ is contracted by $\pi$. But we have $\tilde{Y}.f=3$ and $\pi$ cannot induce a closed embedding on $\tilde{Y}$. If $S$ is a cone over a smooth elliptic curve $E$ of degree $6$ in $\mathbb{P}^5$. Then $S$ is the image of a ruled surface over $E$ by the morphism associated to a linear system numerically equivalent to $|C_0+6f|$, where $C_0$ is the section we obtained by blowing up the vertex ($C_0^2=-6$). Then $\tilde{Y}\equiv mC_0+18f$, where $m=3$ or $4$. Then again if $f$ is the general fibre passing through $O$, $\tilde{Y}.f=m>1$ and this contradicts the fact that $\pi$ induced a closed embedding on $\tilde{Y}$. If $S$ is a del Pezzo surface of degree $6$ with possible isolated double points, then $S$ is the image of $\pP^2$ blowing up three points, i.e. $S$ is the image of the linear system $|3H-E_1-E_2-E_3|$ where $H$ is the pull-back of a hyperplane section in $\pP^2$ and $E_i$ are the exceptional divisors. Note that the singularities occur when the points are not in general position. Then (compactification) of $\pi(S)\subset \mathbb{P}^5$ is the image of the blow-up of $S$ at $O$ by the morphism associated to $|3H-E_1-E_2-E_3-G|$, where $G$ is the exceptional divisor of the blow-up at $O$. Then $\pi(S)\subset \pP^5$ is one of the following:
    \begin{enumerate}[label=(\roman*)]
        \item The Veronese surface $\pP^2\subset \pP^5$.
        \item $\mathbb{F}_0 $ or $\mathbb{F}_2$ embedded into $\pP^5$ as a minimal surface.
        \item A del Pezzo surface of degree $5$ with possibly isolated double points.
    \end{enumerate}
    On the other hand, $\pi(S)\subset V_4$, and hence corresponds to a Cartier divisor of the form $\cO_{V_4}(l)$. Thus $\pi(S)$ has degree $4l$. This will immediately lead to a contradiction in the third case. In the first and second case, we will have $l=1$. But remember that a smooth hyperplane section of $V_4$ can only be a del Pezzo surfaces of degree $4$, which is impossible in the first two cases.\ 
    
    The case when $S$ is of type (3) is more complicated. Note when $k\neq 3$, the secant variety of $\mathbb{F}_{2k}\subset \pP^7$ is a proper subvariety of $\pP^7$ and projection from a point off the secant variety induces an isomorphism. Since a rational normal scroll does not have a tri-secant line (not contained in the surface), any secant line can only meet the scroll at two points transversely. Similarly, a tangent line (not contained in the surface) can be tangent at one point and does not meet the scroll again. Moreover, two distinct secant lines, a tangent line and a secant line or two tangent lines cannot meet at any points other than a point of the scroll unless the scroll intersects the plane spanned by these two lines in a conic. In conclusion, $S$ can either be smooth, or have only one double point or only one double line (which comes from a conic on $\mathbb{F}_{2k}$). Now suppose $O\in S$ is a smooth point. Then there exists a unique corresponding point, which we also call $O$ on $\mathbb{F}_{2k}$. Then $\pi(S)$ is the image of a codimension one base-point-free linear system in $|C_0+(3+k)f-E|$, where $E$ is the exceptional divisor obtained by blowing up $O$. Thus $\pi(S)\subset V_4$ is a Cartier divisor of degree $5$, which leads to a contradiction as before.\ 
    
    When $O$ is the double point, then there are either two distinct points or a point and a tangent direction at it on $\mathbb{F}_{2k}$ corresponding to $O$. For simplicity of argument, we will only discuss the case of two distinct points, the other case can be handled similarly. Denote the two points by $p_1,p_2$, then $\pi(S)$ is the image of $\mathbb{F}_{2k}$ via the complete linear system $|C_0+(3+k)f-E_1-E_2|$ where $E_1,E_2$ are the exceptional divisors obtained by blowing up $p_1,p_2$. Note by our choice, $p_1,p_2$ cannot lie on a line contained in $S$ nor a conic in $S$. When $k=0$, there are no conics on $S$. When $k=1$, $C_0$ is the only conic in $S$. When $k=2$, $C_0+f$ will be a conic for any general fibre $f$. When $k=3$, any two points not on the same line will be on a conic (which is the union of the lines containing each point). It is not hard to check that $\pi(S)$ is one of the following surfaces:
    \begin{enumerate}[label=(\Alph*)]
        \item Image of $\mathbb{F}_0$  associated to complete linear system $|C_0+2f|$.
        \item Image of $\mathbb{F}_2$  associated to complete linear system $|C_0+3f|$.
        \item Image of $\mathbb{F}_4$  associated to complete linear system $|C_0+4f|$.
    \end{enumerate}
    We present in Table $1$ what $\pi(S)$ is depending on the position of $p_1,p_2$.
            \begin{table}[ht]\label{tb}
            \caption{}
    \begin{center}
\begin{tabular}{|l|l|l|}
\hline
k & Position of $p_1,p_2$             & $\pi(S)$            \\
\hline
0&$p_1,p_2$ not on a $C_0$&A\\
\hline
0&$p_1,p_2$ on a $C_0$&B\\

\hline
1&$p_1\in C_0$ and $p_2\notin C_0$&B\\
\hline
1&$p_1,p_2\notin C_0$&A\\
\hline
2&$p_1,p_2\notin C_0$&B\\
\hline

\end{tabular}
\end{center}
\end{table}
Note all $\pi(S)$ are smooth surfaces in $\pP^5$ of degree $4$. On the other hand, $\pi(S)\subset V_4$ corresponds to a Cartier divisor $\cO_{V_4}(l)$. By looking at the degree, we see that $l=1$. But then $\pi(S)$ has to be a del Pezzo surface of degree $4$, which leads to a contradiction.\ 

When $O$ is a point on the double line, then $k=1,2$ or $3$. In each case, the strict transform of the conic under the blow-up of $E_1,E_2$ will be contracted by the system $|C_0+(3+k)f-E_1-E_2|$, and it is not hard to check that $\pi(S)$ is of type $(C)$ by computing the degree and the self-intersection of the strict transform for each $k$. Since $\pi(S)=(C)$ has degree $4$, it corresponds to the Cartier divisor $\cO_{V_4}(1)$, which means $\pi(S)\subset H$, where $H$ is a hypersurface in $\pP^5$. This is absurd since $(C)$ is non-degenerate. This finishes the proof.
\end{proof}

\cite{D} classified semistable rank $2$ sheaves with Chern classes $c_1=0$, $c_2=2$  and $c_3=0$ on cubic threefolds. Now we classify semistable rank $2$ sheaves with Chern classes $c_1=0$, $c_2=2$  and $c_3=0$ on $V_4$, closely following the argument of \cite{D}. When the proof transfers almost verbatim, we will only point out the changes in our situation and refer the readers to \cite{D}.

\begin{prop}
  Let $E$ be a semistable rank $2$ sheaf with Chern classes $c_1=0$, $c_2=2$  and $c_3=0$ on $V_4$. Let $F$ be the double dual of $E$. Then either $E$ is locally free or $F$ is locally free with second Chern class $c_2(F)=1$ and $h^0(F)=1$ or $F=H^0(F)\otimes\cO_{V_4}$.
\end{prop}
\begin{proof}
    See \cite[Proposition 3.1]{D}. The proof of  \cite[Proposition 3.1]{D} used mainly general arguments about semistable sheaves on projective varieties and can be directly applied here. We highlight a few similarities between $V_4$ and a cubic threefold $V_3$ which allow us to mimic the argument:
    \begin{enumerate}
        \item Both $V_4$ and $V_3$ are Fano threefolds of index $2$.
        \item General hyperplane sections of both $V_4$ and $V_3$ are del-Pezzo surfaces in their anticanonical embedding.
        \item The following inequalities for Hilbert polynomials on $V_3$
        \begin{align*}
            \chi(E(n))&<2\chi(\cO_{V_3}(n))\\
            \chi(E(n))&<2\chi(I_p(n))
        \end{align*}
    \end{enumerate}
    remain true on $V_4$, in fact
    \begin{align*}
        \chi(E(n))&=\frac{4}{3}n^3+4n^2+\frac{8}{3}n\\
        \chi(\cO_{V_4}(n))&=\frac{2}{3}n^3+2n^2+\frac{7}{3}n+1\\
        \chi(I_p(n))&=\frac{2}{3}n^3+2n^2+\frac{7}{3}n+1
    \end{align*}
    where $p$ is a point.
\end{proof}

\begin{lem}
    Suppose $\theta$ is  the theta-characteristic of a smooth conic $C\subset V_4$. We consider the sheaf $E$ which is the kernel of the surjection $H^0(\theta(1))\otimes\cO_{V_4}\to\theta(1)$. Then $E$ is stable with Chern classes $c_1(E)=0$,$c_2(E)=2$ and $c_3(E)=0$.
\end{lem}
\begin{proof}
    See \cite[Lemma 3.4]{D}. Again the arguments in \cite{D} can be applied because of the similarites highlighted in the previous proof. The only new fact we need is $2\chi(I_C(n))<\chi(E(n))$. This is true since
    \begin{align*}
        \chi(I_C(n))=\frac{2}{3}n^3+2n^2+\frac{1}{3}n.
    \end{align*}
\end{proof}
\begin{thm}
    Let $E$ be a semistable rank $2$ sheaf with Chern classes $c_1=0$, $c_2=2$  and $c_3=0$ on $V_4$. If $E$ is stable, then either $E$ is locally free or $E$ is associated to a smooth conic $Y\subset V_4$ such that we have the exact sequence:
     \begin{align*}
         0\to E\to H^0(\theta(1))\otimes \cO_{V_4}\to \theta(1)\to 0,
     \end{align*}
     where $\theta$ is the theta-characteristic of $Y$.\\
     If $E$ is strictly semistable, then  $E$ is the extension of two ideal sheaves of lines.
\end{thm}
\begin{proof}
    See \cite[Theorem 3.5]{D}. Again the arguments in \cite{D} can be applied due to the highlights in the previous two proofs. The only new fact we need in this proof is 
    \begin{align*}
        \chi(I_Z(n))=\frac{2}{3}n^3+2n^2+\frac{7}{3}n+1-l(Z),
    \end{align*}
    when $Z$ is a zero-dimensional subscheme.
\end{proof}
\section{Relation to semistable rank $2$ bundles on $C$}
We now use the classification of semistable rank $2$ sheaves with Chern classes $c_1=0$, $c_2=2$  and $c_3=0$ to understand their relation with semistable rank $2$ bundles on $C$.
\begin{lem}
    Let $E$ be a semistable rank $2$ sheaf with Chern classes $c_1=0$, $c_2=2$  and $c_3=0$ on $V_4$. Then $E\in \mathcal{B}_{V_4}$.
\end{lem}
\begin{proof}
    It suffices to show that $H^*(E(-1))=H^*(E)=0$. If $E$ is a stable vector bundle, the result follows from Theorem \ref{auto} and \cite[Lemma B.2]{Ku3}.\\
    If $E$ is associated to a smooth conic, we have the short exact sequence:
    \begin{align*}
        0\to E\to H^0(\theta(1))\otimes\cO_{V_4}\to \theta(1)\to 0
    \end{align*}
    Since $H^*(\cO_{V_4}(-1))=H^*(\theta)=0$, we immediately obtain $H^*(E(-1))=0$. On the other hand, we have $H^0(H^0(\theta(1))\otimes\cO_{V_4})=2$ and $H^0(\theta(1))=2$. It is clear that the map $H^0(H^0(\theta(1))\otimes\cO_{V_4}) \to H^0(\theta(1))$ is surjective. Moreover, $H^i(\cO_{V_4})=H^i(\theta(1))=0$ for all $i>0$. Thus $H^*(E)=0$.\\
    If $E$ is the extension of the ideal sheaves of lines in $V_4$, the result follows from the fact that $I_l\in \mathcal{B}_{V_4}$.
\end{proof}
\begin{rem}
    It is worth noting that when $E$ is associated to a smooth conic, the short exact sequence:
    \begin{align*}
         0\to E\to H^0(\theta(1))\otimes\cO_{V_4}\to \theta(1)\to 0
    \end{align*}
    expresses the fact that $E$ is the left mutation $\mathbf{L}_{\cO_{V_4}}(\theta(1))[-1]$. This will imply $E\in \cO_{V_4}^\perp$ (hence $H^*(E)=0$) immediately.
\end{rem}

\cite[Theorem 5.10]{Ku2} proved that for any minimal instanton bundle $E$ on $V_4$, $\Phi^*(E)[-1]$ is a simple rank $2$ degree $0$ vector bundles on $C$. We generalized this result in the following theorem:
\begin{thm}\label{rep}
    Let $E$ be a semistable rank $2$ sheaf with Chern classes $c_1=0$, $c_2=2$  and $c_3=0$ on $V_4$, then $\mathcal{F}\coloneqq\Phi^*(E)[-1]$ is a semistable vector bundle of rank $2$ and degree $0$ on $C$.
\end{thm}

\begin{proof}
    If $E$ is a vector bundle, by \cite[Theorem 5.10]{Ku2}, $\mathcal{F}=\Phi^*(E)[-1]$ is a rank $2$ degree $0$ vector bundle on $C$ such that
    \begin{align*}
        \Hom_C(\mathcal{F},\mathcal{R})=\Ext_C^1(\mathcal{F},\mathcal{R})=0,
    \end{align*}
    where $\mathcal{R}$ is a second Raynaud bundle. By Lemma \ref{scurve}, $\mathcal{F}$ is semistable.\\
    If $E$ is associated to a smooth conic $Y$, we have the exact triangle in $\mathcal{D}^b(V_4)$:
    \begin{align*}
        \theta(1)[-1]\to E\to H^0(\theta(1))\otimes\cO_{V_4}.
    \end{align*}
    Apply the functor $\Phi^*(\cdot)[-1]$ and noting $\Phi^*(\cO_{V_4})=0$, we obtain
    \begin{align*}
        \mathcal{F}= \Phi^*(\theta(1))[-2]
    \end{align*}
    Recall as pointed out in the proof of \cite[Theorem 5.10]{Ku2}, $\Phi^*$ is a Fourier-Mukai transform with the kernel $\mathcal{S}^*\otimes p_{V_4}^*\cO_{V_4}(-2)[3]$. Thus the fiber of the object $\mathcal{F}$ at a point $x\in C$ is given by 
    \begin{align*}
        \mathcal{F}_x&=H^{\bullet+1}(V_4,\mathcal{S}_x^*\otimes \theta(-1))\\
        &=H^{\bullet+1}(Y,\mathcal{S}_x^*|_Y\otimes \theta(-1))\\
        &=H^{-\bullet}(\pP^1,\mathcal{S}_x|_{Y}\otimes\cO_{\pP^1}(1))^*
    \end{align*}
    Now $\mathcal{S}_x|_Y$ is a rank $2$ bundle on $Y\simeq \pP^1$ with degree $-2$. Moreover, we note $H^0(V_4,\mathcal{S}^*_x)=\C^4$ and the induced map $\cO_{V_4}^{\oplus 4}\to \mathcal{S}^*_x$ is surjective (see \cite[Proposition 5.7]{Ku2}). Hence $\mathcal{S}^*_x|_Y$ as a sheaf on $Y$ is generated by global sections. Thus
    \begin{align*}
        \mathcal{F}_x=\C^2[0]
    \end{align*}
    for all $x\in C$. Hence $\mathcal{F}$ is a vector bundle of rank $2$. By Lemma \ref{RR}, we see that $\mathcal{F}$ has degree $0$.\\
    
   It remains to see that $\mathcal{F}$ is semistable. Note that a vector bundle $\mathcal{F}$ of rank $2$ and degree $0$ is not stable if and only if there is a nontrivial morphism $\mathcal{F}\to \mathcal{L}$, where $\mathcal{L}$ is a line bundle of degree $0$. By adjunction
    \begin{align*}
        \Hom(\mathcal{F},\mathcal{L})=\Hom(\Phi^*(\theta(1))[-2],\mathcal{L})&=\Hom(\theta(1),\Phi(\mathcal{L})[2])\\
        &=\Hom(\theta(1),I_l[1])\\
        &=\Ext^1(\theta(1),I_l)
    \end{align*}
    where $l$ is a line on $V_4$ by Theorem \ref{ideal}. Apply $\Ext^\bullet(\theta(1),-)$ to the short exact sequence $0\to I_l\to \cO_{V_4}\to \cO_l\to 0$, we have
    \begin{align*}
        \Hom(\theta(1),\cO_l)\to \Ext^1(\theta(1),I_l)\to \Ext^1(\theta(1),\cO_{V_4})
    \end{align*}
    Now the first space is zero since $\theta$ is supported on a smooth conic, while the last space is Serre dual to $\Ext^2(\cO_{V_4}(2),\theta(1))=H^2(\theta(-1))=0$. Thus $\Ext^1(\theta(1),I_l)=0$ and $\mathcal{F}$ is in fact stable in this case.

    If $E$ is the extension of ideal sheaves of lines in $V_4$, then we have the short exact sequence:
    \begin{align}
        0\to I_{l_1}\to E\to I_{l_2}\to 0
    \end{align}
    Applying the functor $\Phi^*(\cdot)[-1]$, we have exact triangle
    \begin{align}\label{aa}
        \Phi^*(I_{l_1})[-1]\to \mathcal{F}\to \Phi^*(I_{l_2})[-1]
    \end{align}
    By \cite[Lemma 5.5]{Ku2},  $\Phi^*(I_{l_i})[-1]$ are line bundles of degree $0$ on $C$, thus $\mathcal{F}$ is a strictly semistable rank $2$ bundle of degree $0$.
\end{proof}

\begin{prop}\label{ss}
  Let $E$ be a stable(strictly semistable) rank $2$ sheaf with Chern classes $c_1=0, c_2=2$ and $c_3=0$ on $V_4$, then $\Phi^*(E)[-1]$ is a stable(strictly semistable) vector bundle on $C$.
\end{prop}
\begin{proof}
    By the above theorem, $\Phi^*(E)[-1]$ is always semistable. By the proof of the above theorem, we see that for a strictly semistable instanton $E$, $\Phi^*(E)[-1]$ is a strictly semistable vector bundle on $C$.
    Suppose for a semistable rank $2$ sheaf with Chern classes $c_1=0$, $c_2=2$  and $c_3=0$ $E$, $\mathcal{F}=\Phi^*(E)[-1]$ is a strictly semistable vector bundle on $C$. Let
    \begin{align*}
        0\to \mathcal{L}_1\to\mathcal{F}\to \mathcal{L}_2\to 0
    \end{align*}
    be a Jordan-H\"older filtration. Then $\mathcal{L}_1,\mathcal{L}_2$ have to be degree $0$ line bundles. Apply the functor $\Phi(\cdot)[1]$, by \cite[Lemma 5.5]{Ku2}, there exist two lines $l_1,l_2$ in $V_4$ such that 
    \begin{align*}
        0\to I_{l_1}\to E\to I_{l_2}\to 0
    \end{align*}
    is exact. Hence $E$ is a strictly semistable sheaf.
    
\end{proof}
By now we have established a well-behaved correspondence between (semi)stable rank $2$ sheaves with Chern classes $c_1=0$, $c_2=2$  and $c_3=0$ on $V_4$ and (semi)stable rank $2$ degree $0$ vector bundles on $C$. Next is to use this correspondence to analyze the two moduli spaces.

\section{Moduli Space of Instantons}
We start by showing the smoothness of $M_{V_4}$. To do this we first compute some related invariants.
\begin{lem}\label{smooth1}
    
    Let $\theta$ be the theta characteristic of a smooth conic $Y$ in $V_4$. Let $E$ be the kernel of the natural surjection $H^0(\theta(1))\otimes\cO_{V_4}\to \theta(1)$. Then $\Ext^2(E,E)=0$ and $\Ext^1(E,E)$ has dimension $5$.
\end{lem}
\begin{proof}
    By Theorem \ref{rep}, $E=\Phi(\mathcal{F})[1]$, where $\mathcal{F}$ is a rank $2$ bundle on $C$. Thus
    \begin{align*}
        \Ext^2(E,E)=\Ext^2(\Phi(\mathcal{F})[1],\Phi(\mathcal{F})[1])=\Ext_C^2(\mathcal{F},\mathcal{F}).
    \end{align*}
    The last space is $0$ since $C$ is a curve.
   
    Now $\Ext^3(E,E)\simeq\Hom(E,E(-2))^*=0$ and $\Hom(E,E)=\C$. By Riemann-Roch, $\chi(E,E)=-4$. Thus $\Ext^1(E,E)$ is five dimensional.
\end{proof}

\begin{lem}\label{smooth 2}
    Let $l_1,l_2\subset V_4$ be two lines. Then $\Ext^2(I_{l_1},I_{l_2})=0$ and $\dim\Ext^1(I_{l_1},I_{l_2})=1$ if $l_1\neq l_2$ and $2$ if $l_1=l_2$.
\end{lem}
\begin{proof}
    Since $I_{l_1}, I_{l_2}\in \mathcal{B}_{V_4}$, we have 
    \begin{align*}
        \Ext^2(I_{l_1}, I_{l_2})&=\Ext_{\mathcal{D}^b(C)}^2(\Phi^{-1}(I_{l_1}[-1]),\Phi^{-1}(I_{l_2}[-1]))\\
        &=\Ext^2_C(\mathcal{L}_1,\mathcal{L}_2)
    \end{align*}
    where $\mathcal{L}_i$ is the  line bundle corresponding to $l_i$ as in Theorem \ref{ideal}. Since $C$ is a curve, we see that the above extension group is $0$.\\
    Moreover, $\Ext^3(I_{l_1},I_{l_2})\simeq\Hom(I_{l_2},I_{l_1}(-2))^*=0$. By Riemann-Roch, $\chi(I_{l_1},I_{l_2})=-1$. Thus the lemma follows.
\end{proof}

Let $N\geq 1$ be an integer and $V$ a complex vector space. Let $Q$ be the Quot scheme parametrizing quotients $V\otimes \cO_{V_4}(-N)\to E$ of rank $2$ on $V_4$ and Chern classes $c_1(E)=0$, $c_2(E)=2$ and $c_3(E)=0$. Let $L$ denote the natural polarization on $Q$ (See \cite[Section 1]{Si}). The group $G=PGL(V)$ acts on $Q$ and a suitable power of $L$ is $G$-linearized. Let $Q_c^{ss}$ be the $PGL(V)$-semistable points corresponding to quotients without torsion and $Q_c$ the closure of $Q_c^{ss}$ in $Q$. When the integer $N$ and
the vector space V are suitably chosen the following properties are satisfied. The map $V\otimes \cO_{V_4}\to E(N)$ induces an isomorphism $V\simeq H^0(E(N))$ and $h^i(E(k))=0$ for $k\geq N$ and $i\geq 1$ and for all $E$ in $Q_c$. The point $[E]\in Q_c$ is semistable if and only if the sheaf $E$ is semistable if and only if $E\in Q_c^{ss}$. The stabilizer of $[E]$ in $GL (V)$ is identified with the group of
automorphisms of the sheaf $E$ and the moduli space is then the GIT quotient $Q_c^{ss}//G$.
\begin{lem}\label{smooth3}
With the above hypotheses, the scheme $Q^{ss}_c$ is smooth.
\end{lem}
\begin{proof}
    The tangent space of $Q^{ss}_c$ at a point $[E]$ is isomorphic to $\Hom(F,E)$ where $F$ is the kernel of the map $V\otimes \cO_{V_4}(-N)\to E$. The scheme $Q_c^{ss}$ is smooth at the point if $\Ext^1(F,E)=0$. Consider the exact sequence:
    \begin{align*}
        \Ext^1(V\otimes \cO_{V_4}(-N),E)\to \Ext^1(F,E)\to \Ext^2(E,E)
    \end{align*}
    We then obtain an inclusion $\Ext^1(F,E)\to \Ext^2(E,E)$ since $h^1(E(N))=0$. It suffices then to prove that $\Ext^2(E,E)=0$. But this follows from the fact $E=\Phi(\mathcal{F})[1]$ where $\mathcal{F}$ is a sheaf on $C$, as we have seen in the proof of Lemma \ref{smooth1} and \ref{smooth 2}.
\end{proof}
\begin{thm}\label{smooth}
     The moduli space $M_{V_4}$ of semistable sheaves of rank $2$ with Chern classes $c_1(E)=0,c_2(E)=2,c_3(E)=0$ on $V_4$ is smooth of dimension $5$.
\end{thm}
\begin{proof}
    See \cite[Theorem 4.6]{D}. Let $x\in Q^{ss}_c$ and $E$ be the corresponding sheaf. Let $Q^{s}_c\subset Q_c$ be the set of stable  sheaves and $M^{s}$ be the moduli space of stable sheaves. The scheme $Q^{s}$ is a principal $G$-space over $M^{s}$ and hence $M^{s}$ is smooth by the above lemma. It remains to study the case when $E=I_{l_1}\oplus I_{l_2}$ where $l_1,l_2$ are two lines in $V_4$. The orbit $O(x)$ of $x$ under $G$ is closed. The stabilizer $G_x$ of $x$ is a reductive group and there exists an affine subscheme  $W\subset Q^{ss}_c$ containing $x$ which is locally closed and stable under $G_x$ so that the morphism $W//G_x\to Q^{ss}_c//G$ is \'etale(\cite{L}). Let $(W, x)$ be the germ from $W$ to $x$ and let $F$ be the restriction to $X\times (W, x)$ of the tautological quotient on $X \times Q$. Then $((W,x),F)$ is a space of versal deformation for the sheaf $E$.(\cite{Og} Proposition 1.2.3). The germ $W$ is then smooth at $x$ by Lemma \ref{smooth 2}. Since the morphism $W//G_x\to Q^{ss}_c//G$ is \'etale, it suffices to show the quotient $W//G_x$ is smooth at $[x]$. Now there exists a $G_x$ linear morphism $(W,x)\to (T_xW,0)$ \'etale at $x$ so that the induced morphism $W//G_x\to T_xW//G_x$ is \'etale at $[x]$(\cite{L}). It is therefore sufficient to prove that the quotient $T_xW//G_x$ is smooth at 0.\
    
    Suppose $l_1$ and $l_2$ are distinct. The tangent space $T_xW=\Ext^1(E,E)$ is of dimension $6$ and $G_x=G_m\times G_m$ acts on the space by formula (\cite{Og} Lemma 1.4.16)
    \begin{align*}
        (t,t')\cdot(\sum_{i,j}e_{i,j})=e_{1,1}+t/t'e_{1,2}+t'/te_{2,1}+e_{2,2}.
    \end{align*}
    It is easy to verify the quotient $T_xW//G_x$ is the affine space $\mathbb{A}^5$ and in particular smooth at $0$.\
    
    Suppose $l_1$ and $l_2$ are the same and use $l$ to denote the line. The tangent space $\Ext^1(E,E)$ is then of dimension $8$ and $G_x=\mathrm{PGL}(2)$. Let $T=\Ext^1(I_l,I_l)$ and let $U$ be a vector space of dimension $2$. The group $G_x$ acts on $T_xW=T\otimes \mathrm{End}(U)$ by conjugation on $\mathrm{End}(U)$(\cite{Og} Lemma 1.4.6) The quotient $T_x//G_x$ is then isomorphic to $\mathbb{A}^5$(\cite{La}, III, case 2) and in particular smooth at $0$.
\end{proof}
We now construct a morphism from $M_{V_4}$ to $M_C$ ,the moduli space of semistable vector bundles of rank $2$ and degree $0$ on $C$. $M_{V_4}$ is the GIT quotient  $Q^{ss}_c//G$. Let $\mathcal{E}$ be a universal family on $Q^{ss}_c\times V_4$. For any $t\in Q^{ss}_c$, by Lemma \ref{rep} and Lemma \ref{ss}, the map
\begin{align*}
    \Psi:Q^{ss}_c&\to M_C\\
    t&\mapsto [\Phi^*(\mathcal{E}_t)[-1]]
\end{align*}
is well defined. Since $\Phi^*(-)[-1]$ is Fourier-Mukai, $\Psi$ is algebraic. To see that this morphism is invariant under $G$, it suffices to check that $\Psi(t)=\Psi(t_0)$ for any $t$ representing a strictly semistable sheaf $\mathcal{E}_t$, an extension of $I_{l_1}$ and $I_{l_2}$, with $t_0$ representing $I_{l_1}\oplus I_{l_2}$. Applying the functor $\Phi^*(-)[-1]$ to the short exact sequence
\begin{align*}
    0\to I_{l_1}\to E\to I_{l_2}\to 0
\end{align*}
we obtain 
\begin{align*}
    0\to \Phi^*(I_{l_1})[-1]\to \Phi^*(E)[-1]\to \Phi^*(I_{l_2})[-1]\to 0
\end{align*}
Recall that $\Phi^*(I_{l_i})[-1]$ are line bundles of degree $0$ thus $\Phi^*(E)[-1]$ lies in the $S$-equivalence class of $\Phi^*(I_{l_1})[-1]\oplus\Phi^*(I_{l_2})[-1]$.
As a result $\Psi$ descends to a morphism $\psi:M_{V_4}\to M_C$.
\begin{thm}
    $\psi\colon M_{V_4}\to M_C$ is an isomorphism. As a result, the moduli space of semistable rank $2$ sheaves with Chern classes $c_1=0$, $c_2=2$  and $c_3=0$ on $V_4$ is a projective bundle over the Jacobian of $C$.
\end{thm}
\begin{proof}
    Since both $M_{V_4}$ and $M_C$ are projective, $\psi$ is proper. We claim $\psi$ is injective. Let $\psi([E_1])=\psi([E_2])$. By Proposition \ref{ss}, either both $E_i$ are stable or both $E_i$ are strictly semistable. If $E_i$ are stable, then $\psi([E_1])=\psi([E_2])$ implies $E_1\simeq E_2$, i.e. $[E_1]=[E_2]$. If $E_i$ are strictly semistable, the injectivity follows from Theorem \ref{ideal}. 
    
    Any proper quasi-finite morphism is finite, so $\psi$ is a finite morphism. Thus the image $\psi(M_{V_4})$ has dimension $5$, which must be all of $M_C$. So $\psi$ is surjective. Since $\psi$ is injective and $M_C$ is integral, we see that $M_{V_4}$ must be connected.
    Along with Theorem \ref{smooth}, we know that $M_{V_4}$ is a smooth variety.  
    
    Let $[E]\in M_{V_4}$ be a stable point, then the tangent space at $[E]$ is given by $\Ext^1(E,E)$. The tangent space at $\psi(E)$ (which is also stable) is given by $\Ext^1(\Phi^*(E)[-1],\Phi^*(E)[-1])$.
    Since $\Phi^*:\mathcal{B}_{V_4}\to\mathcal{D}^b(C)$ is an equivalence, $\psi$ induces isomorphism between tangent spaces, so $\psi$ is \'etale when restricted to the open stable locus. Since $\psi$ is also injective and we are working over $\C$, $\psi$ is an open immersion over the stable locus (See for example Stack Project $40.14$). Now $\psi$ is a bijective birational proper morphism, it has to be an isomorphism.

\end{proof}

\end{document}